\documentclass[12pt]{amsart}

\newtheorem{theorem}{Theorem}[section]
\newtheorem{proposition}[theorem]{Proposition}
\newtheorem{lemma}[theorem]{Lemma}

\numberwithin{equation}{section}

\begin{document}
\baselineskip=15.5pt

\title[Poincar\'e sheaf over irreducible curves]{Poincar\'e sheaves
on the moduli spaces of 
torsionfree sheaves over an irreducible curve}

\author[U. N. Bhosle]{Usha N. Bhosle}

\address{School of Mathematics, Tata Institute of Fundamental
Research, Homi Bhabha Road, Mumbai 400005, India}

\email{usha@math.tifr.res.in}

\author[I. Biswas]{Indranil Biswas}

\address{School of Mathematics, Tata Institute of Fundamental
Research, Homi Bhabha Road, Mumbai 400005, India}

\email{indranil@math.tifr.res.in}



\date{}

\begin{abstract} 
Let $Y$ be a geometrically irreducible reduced projective curve
defined over $\mathbb R$. Let $U_Y(n,d)$ (respectively, $U'_Y(n,d)$)
be the moduli space of geometrically stable
torsionfree sheaves (respectively, locally free sheaves) on $Y$
of rank $n$ and degree $d$. Define $\chi\, =\, d+n(1-\text{genus}(Y))$,
where $\text{genus}(Y)$ is the arithmetic genus.
If $2n$ is coprime to $\chi$, then there
is a Poincar\'e sheaf over $U_Y(n,d)\times Y$.
If $2n$ is not coprime to $\chi$, then there
is no Poincar\'e sheaf over any nonempty open subset of $U'_Y(n,d)$.
\end{abstract}

\maketitle

\section{Introduction}

Let $C$ be a smooth complex projective curve.
Let $U_C(n,d)$ be the moduli space of stable vector bundles over $C$
of rank $n$ and degree $d$. Assume that $U_C(n,d)$ is nonempty;
this is ensured if $\text{genus}(C)\, \geq\, 2$. A vector bundle
${\mathcal E}\,\longrightarrow\, U_C(n,d)\times C$ is called
a \textit{Poincar\'e vector bundle} if for each point $z\,\in\,
U_C(n,d)$, the vector bundle ${\mathcal E}\vert_{\{z\}\times C}$
over $C$ is in the isomorphism class defined by $z$. By a
\textit{Poincar\'e vector bundle} over an open subset
${\mathcal U}\, \subset\, U_C(n,d)$ we will mean a vector
bundle ${\mathcal E}\,\longrightarrow\, {\mathcal U}\times C$
such that for each point $z\,\in\, {\mathcal U}$, the vector
bundle ${\mathcal E}\vert_{\{z\}\times C}$
is in the isomorphism class defined by $z$.

If $n$ is coprime to $d$, then there is a Poincar\'e
vector bundle over $U_C(n,d)\times C$. If $n$ and $d$ are not coprime,
then there is no Poincar\'e
vector bundle over any nonempty open subset of $U_C(n,d)$
\cite[Theorem 2]{Ra}. We note that $n$ and $d$
are coprime if and only if $n$ and $\chi\, :=\, d- n(\text{genus}(C)
-1)$ are coprime; the Euler characteristic of any vector bundle of rank 
$n$ and degree $d$ over $C$ is $\chi$.

Let $D$ be a geometrically irreducible smooth projective curve
defined over $\mathbb R$. A vector bundle
$E$ over $D$ is called geometrically stable if the vector bundle
$E\bigotimes_{\mathbb R}{\mathbb C}$ over $D\times_{\mathbb R}
\mathbb C$ is stable.
Let $U_D(n,d)$ be the moduli space of geometrically stable vector 
bundles over $C$ of rank $n$ and degree $d$.
Assume that $U_D(n,d)$ is nonempty;
this is ensured if $\text{genus}(D)\, \geq\, 2$.
If $D$ has a real point,
then there is a Poincar\'e
vector bundle over $U_D(n,d)\times D$ if and only if $n$ and $\chi$
are coprime, where $\chi$ is defined as above. If $D$ does not have
any  real point, then there is a Poincar\'e
vector bundle over $U_D(n,d)\times D$ if and only if $2n$ and $\chi$
are coprime. (See \cite{BHu}.)

Our aim here is to address this question for curves not
necessarily smooth.

Let $Y$ be a geometrically irreducible reduced projective curve
defined over
the real numbers. A torsionfree sheaf $V$ on $Y$ is called geometrically
stable if the coherent sheaf on $Y\times_{\mathbb R}\mathbb C$ defined
by $V$ is stable. Let $U_Y(n,d)$
(respectively, $U'_Y(n,d)$) be the moduli space of geometrically stable
torsionfree sheaves (respectively, locally free sheaves) on
$Y$ of rank $n$ and degree $d$. We assume that the moduli
space $U_Y(n,d)$ has points defined over $\mathbb C$ (the set
of points of $U_Y(n,d)$ defined over $\mathbb R$ is allowed to
be empty); this is ensured if the arithmetic genus of $Y$
is at least two.

Define $\chi\, :=\, d- n(\text{genus}(Y)-1)$, which is the Euler
characteristic of any sheaf on $Y$ lying in $U_Y(n,d)$.

Assume that $Y$ does not have any point defined over $\mathbb R$..
We prove the following (see Theorem \ref{thm2}):

\medskip
\textit{If $2n$ is coprime to $\chi$, then there
is a Poincar\'e sheaf over $U_Y(n,d)\times Y$.}

\textit{If $2n$ is not coprime to $\chi$, then there is no
Poincar\'e sheaf over any nonempty open subset of $U'_Y(n,d)$.}
\medskip

The proof of Theorem \ref{thm2} given here is different from the
one in \cite{BHu} where it is proved under the assumption that
$Y$ is smooth. The proof in \cite{BHu} uses the smoothness
assumption in an essential way.

\section{Poincar\'e sheaf}\label{sec1}

Let $X$ be an irreducible reduced projective curve defined
over an algebraically closed field $k$ of characteristic zero. 
Let $U(n,d)$ be the moduli space of torsionfree
stable sheaves of rank $n$ and degree $d$ on $X$; it is 
a GIT quotient of a Quot scheme $Q$ of coherent quotient sheaves of 
${\mathcal O}_X^N$ 
by $\text{PGL}(N,k)$ \cite{Ne}, \cite{S}. We assume that
$U(n,d)$ has points defined over $\mathbb C$; as mentioned before,
the set of points of $U_Y(n,d)$ defined over $\mathbb R$ may be
be empty.

Let $Q'$ denote the set of points of $Q$ corresponding to stable 
quotient sheaves. Over $Q'\times X$, there is a universal sheaf 
$${\mathcal E} \,\longrightarrow\, Q'\times X \, .$$ 
Let ${\mathcal E}_q\,:=\, {\mathcal E}\vert_{q\times X}$. 

The points of $Q'\, \subset\, Q$ are identified with the properly
stable points for the action of $\text{PGL}(N,k)$ on $Q$.
The isotropy subgroup $G_q$ at $q$ is isomorphic to 
$\text{Aut}({\mathcal E}_q)$. 
Since ${\mathcal E}_q$ is stable, we have $G_q \,=\, k^*$; 
any $c\,\in\, k^*$ acts on ${\mathcal E}$ by multiplication with the 
scalar $c$. By a result of Nevins \cite[p. 2482, Theorem 1.2]{N}, the
sheaf ${\mathcal E}$ descends to $U(n,d)\times X$ 
if and only if for every $(q,x)\in Q'\times X$, the ${\mathcal 
O}_{Q'\times X, (q,x)}$--modules 
$$
{\mathcal E}\otimes ({\mathcal O}_{Q'\times X}/m_{q,x}) ~\,~ \,
\text{~and~}~\,~\,
Tor^{{\mathcal O}_{Q'\times X}}_1({\mathcal E}, {\mathcal O}_{Q'\times 
X}/ m_{q,x})
$$ 
are trivial representations of $G_q$ (the action of $G_q$ on $X$ is 
taken to be the trivial one); here $m_{q,x}$ denotes
the ideal sheaf of the point $(q,x) \,\in\, Q'\times X$.

The proof of Proposition \ref{p1} is straightforward.

\begin{proposition}\label{p1}
Any $c \in k^*$ acts on the modules
$${\mathcal E}\otimes ({\mathcal O}_{Q'\times X/m_{q,x}})~\,~ \,
\text{~and~}~\,~ \,
Tor^{{\mathcal O}_{Q'\times X}}_1({\mathcal E}, {\mathcal O}_{Q'\times
X}/m_{q,x})$$
by multiplication by the scalar $c$.
\end{proposition}

\begin{proposition}\label{thm1}
\begin{enumerate}
	\item There exists a Poincar\'e sheaf on 
$U(n,d)\times X$ if $n$ and $d$ are coprime.
	\item There is no Poincar\'e sheaf on any open subset of 
$U'(n,d)$ if $n$ and $d$ are not coprime.
\end{enumerate}
\end{proposition}

\begin{proof}
The first part is proved in \cite[Ch. 5, \S~7, Theorem 5.12$'$]{Ne}.
We include a proof which will be referred in Section \ref{sec3}.

Let $\chi:= \chi(E_q)$, where $q\in Q'$. Since g.c.d. $(n,d)=1$, 
there 
exist integers $a$ and $b$ such that $n a + \chi b =-1$. 
Let ${\rm Det}({\mathcal E})$ be the determinant line bundle
on $Q'$ associated to the family 
${\mathcal E} \,\longrightarrow\, Q'\times X$. We recall that
$$
{\rm Det}({\mathcal E})\, =\, (\det R^0f_*{\mathcal E})\otimes
(\det R^1f_*{\mathcal E})^*\, ,
$$
where $f\, :\, Q'\times X\,\longrightarrow\, Q'$ is the projection.
Fix a smooth point $x_0\in X$. Define the line bundle on $Q'$ 
\begin{equation}\label{cL}
{\mathcal L}:= ({\rm Det}({\mathcal E}))^b \otimes 
(({\Lambda}^n {\mathcal E}\mid_{Q'\times x_0})^{\otimes a})\, .
\end{equation}
Any $c\,\in\, k^*$ acts on ${\mathcal L}$ as multiplication by $c^{-1}$. 
Define
$$
{\mathcal E}':= {\mathcal E}\otimes p_{Q'}^*{\mathcal L} \, .
$$
Then, by Proposition \ref{p1}, the group $k^*$ acts trivially on 
$$
{\mathcal E}'\otimes ({\mathcal O}_{Q'\times X}/m_{q,x}) ~\,~\,
\text{~and~}~\,~\, 
Tor^{{\mathcal O}_{Q'\times X}}_1({\mathcal E}', {\mathcal O}_{Q'\times 
X}/m_{q,x})\, .
$$
By a result of Nevins \cite[Theorem 1.2]{N}, the sheaf ${\mathcal E}'$ 
descends to $U(n,d)\times X$ giving the required Poincar\'e sheaf.

Th second part follows exactly as in \cite[Corollary 2.3]{BH}.
\end{proof}

\section{Curves defined over real numbers}\label{sec3}
 
Let $Y$ be a geometrically irreducible reduced curve defined over 
$\mathbb R$.
Let $U_Y(n,d)$ be the moduli space of geometrically stable torsionfree 
sheaves on $Y$ of rank $n$ and degree $d$. We assume that $U_Y(n,d)$
is nonempty. Let
$$
U'_Y(n,d)\, \subset\, U_Y(n,d)
$$
be the Zariski open subscheme parametrizing the locally free
stable sheaves.

\begin{lemma}\label{lem1}
Assume that $Y$ has a smooth real point. Then there is a
Poincar\'e sheaf on $U_Y(n,d)\times Y$ if $d$ is coprime
to $n$. If $d$ is not coprime to $n$, then
there is no Poincar\'e sheaf on any nonempty Zariski
open subset of $U'_Y(n,d)$.
\end{lemma}

\begin{proof}
Let $X\,:=\, Y\times_{\mathbb R}\mathbb C$ be the complex curve
obtained by base change to $\mathbb C$. We note that the base
change $U_Y(n,d)_{\mathbb C}\,=\, U_Y(n,d)\times_{\mathbb R}\mathbb C$
is the moduli space $U_X(n,d)$ of stable torsionfree
sheaves on $X$ of rank $n$ and degree $d$. Similarly,
$U'_Y(n,d)_{\mathbb C}\,=\, U'_Y(n,d)\times_{\mathbb R}\mathbb C$
is the moduli space $U'_X(n,d)$ of stable vector bundles
on $X$ of rank $n$ and degree $d$.

If ${\mathcal E}\, \longrightarrow\, {\mathcal U}\times Y$ is a
Poincar\'e sheaf, where ${\mathcal U}\, \subset\, U'_Y(n,d)$
is a nonempty Zariski open subset, then ${\mathcal 
E}\otimes_{\mathbb R}\mathbb C$ is a Poincar\'e
sheaf on $({\mathcal U}\times_{\mathbb R}\mathbb C)\times X$. In that 
case, Proposition \ref{thm1}(2) says that $d$ is coprime to $n$.

Conversely, if $d$ is coprime to $n$, taking the point $x_0$ in
the proof of Proposition \ref{thm1}(1) to be a real point we see that 
the Poincar\'e sheaf constructed in the proof of
Proposition \ref{thm1} is defined over $\mathbb R$.
\end{proof}

Henceforth, we assume that $Y$ does not have any real point.

As in Section \ref{sec1}, define
\begin{equation}\label{chi}
\chi\, :=\, d+n(1-\text{genus}(Y))\, ,
\end{equation}
where $\text{genus}(Y)$ is the arithmetic genus; note that
$\chi\,=\, \chi(E)$ for any $E\,\in\, U_Y(n,d)$.

\begin{theorem}\label{thm2}
Assume that $Y$ does not have any point defined over $\mathbb R$.

\begin{enumerate}
\item If $2n$ is coprime to $\chi$, then there
is a Poincar\'e sheaf over $U_Y(n,d)\times Y$.

\item If $2n$ is not coprime to $\chi$, then there
is no Poincar\'e sheaf over any nonempty open subset
of $U'_Y(n,d)$.
\end{enumerate}
\end{theorem}

\begin{proof}
For a variety $Z$ defined over $\mathbb C$, let
$$
\overline{Z}\, :=\, (Z({\mathbb C})\, ,\overline{\mathcal O}_Z)
$$
be the complex conjugate variety of $Z$.
As in the proof of Lemma \ref{lem1}, define
$X\,:=\, Y\times_{\mathbb R}\mathbb C$. Let
\begin{equation}\label{r-si}
\sigma\, :\, X\, \longrightarrow\, \overline{X}
\end{equation}
be the natural isomorphism obtained from the
fact that $X$ is the base change of $Y$ to $\mathbb C$. The
composition
\begin{equation}\label{r-s2}
X\,\stackrel{\sigma}{\longrightarrow}\,\overline{X}
\,\stackrel{\overline\sigma}{\longrightarrow}\,
\overline{\overline X}\,=\, X
\end{equation}
is the identity map of $X$.

First assume that $2n$ is coprime to $\chi$. Fix a smooth effective 
real divisor $D$ on $Y$ of degree two. We note that such divisors
exist; they are in bijective correspondence with the pairs of smooth 
points of $X$ of the form $\{x\, ,\sigma(x)\}$, where $\sigma$ is the 
map in \eqref{r-si}.

That there is a Poincar\'e sheaf over $U_Y(n,d)\times Y$
can be shown exactly as done in the proof of the first part
of Proposition \ref{thm1}. Instead of the smooth point $x_0$
in the proof of Proposition \ref{thm1}, we take the
above divisor $D$. More
precisely, replace ${\Lambda}^n {\mathcal E}\mid_{Q'\times x_0}$
in \eqref{cL} by the line bundle ${\Lambda}^{2n} {\mathcal 
E}\vert_{Q'\times D}$; note that ${\mathcal
E}\vert_{Q'\times D}$ is a vector bundle of rank $2n$ over
$Q'$. Since $2n$ is coprime to $\chi$ the rest
of the argument remains unchanged.

Now assume that $\chi$ is not coprime to $2n$.

The Euler characteristic $\chi$ is coprime to $n$ if and only if $d$
is coprime to $n$. Therefore,
if $\chi$ is not coprime to $n$, then from Proposition \ref{thm1} we 
know that there is no Poincar\'e sheaf on any nonempty Zariski open
subset of $U'_Y(n,d)$; note that
any Poincar\'e sheaf on ${\mathcal U}\times Y$ defines a Poincar\'e
sheaf on $({\mathcal U}\times_{\mathbb R}{\mathbb C}) \times X$ by base 
change to $\mathbb C$.

Consequently, we assume $\chi$ is coprime to $n$.

Since $\chi$ is not coprime to $2n$, we conclude that $\chi$ is
even. Therefore, $n$ is odd because $\chi$ is coprime to $n$.
Define
\begin{equation}\label{b}
b_0\,:=\, (n-1)/2\,\in\, {\mathbb Z}\, .
\end{equation}

There is an integer $\delta$ and a real point
$$
L_0\, \in\, \text{Pic}^\delta(Y)
$$
such that $L_0$ is not a line bundle over $Y$ \cite[p. 226, Corollary 
2]{BK} (see \cite[p. 159, Proposition 2.2(2)]{GH} for smooth $Y$).
Since there are line bundles
on $Y$ of degree two (recall that there is a smooth
real divisor of degree two),
for any integer $m$, we have the real point
$$
L_0\otimes \xi^{\otimes m}\, \in\, \text{Pic}^{\delta+2m}(Y)\, ,
$$
where $\xi\,\longrightarrow\, Y$ is a line bundle of degree two, 
which is not
a line bundle over $Y$.

We will consider $L_0$ as a line bundle over $X$, because
it defines a point of $\text{Pic}^{\delta+2m}(X)$. Note that
$H^i(X, \, L_0)$ is of even dimension because it
has a quaternionic structure. Hence $\chi(L_0)$ is even.
Therefore, for any even integer $b$, there is a real point
of $\text{Pic}^{b-1+\text{genus}(Y)}(Y)$ which is not
a line bundle over $Y$; note that the Euler characteristic
is $b$.

We may assume that the degree $d$ is sufficiently large positive
by tensoring with the line bundle ${\mathcal O}_Y(aD)$, where
$D$ as before is a real smooth effective divisor of degree two,
and $a\, \in\, \mathbb N$. Note that $\chi$
is also sufficiently large positive because
$\chi\,=\, d-n(\text{genus}(Y)-1)$.

We noted earlier that $n$ is odd, and
$\chi$ is even. Hence
\begin{equation}\label{pair}
\chi+n(\text{genus}(Y)-1)+1-\text{genus}(Y)\,\equiv\, 0~\, ~
\text{~mod~} \, 2\, .
\end{equation}
We also noted above that for any even integer $b$, there is a
real point of $\text{Pic}^{b-1+\text{genus}(Y)}(Y)$ which is not
a line bundle over $Y$. Hence from \eqref{pair} we conclude that
there is a real point
\begin{equation}\label{fL}
L\, \in\, \text{Pic}^{\chi+n(\text{genus}(Y)-1)}(Y)\,=\,
\text{Pic}^{d}(Y)
\end{equation}
which is not a real line bundle over $Y$. Fix such a point $L$.

The line bundle over $X$ (respectively, $\overline{X}$)
(see \eqref{r-s2}) corresponding to $L$ in \eqref{fL} will be 
denoted by $L$ (respectively, $\overline{L}$). Since $L$ in \eqref{fL} 
is a real point, but not a line bundle on $Y$, there is a unique
isomorphism
\begin{equation}\label{eta}
\eta\, :\, L\, \longrightarrow\, \sigma^*\overline{L}
\end{equation}
such that $(\sigma^*\overline{\eta})\circ\eta\,=\, -\text{Id}_L$.
(see \eqref{r-s2}).

Let
\begin{equation}\label{r-s3}
T\, :\, {\mathcal O}_X\oplus {\mathcal O}_X\,\longrightarrow\,
{\mathcal O}_{\overline X}\oplus {\mathcal O}_{\overline X}
\end{equation}
be the isomorphism defined by $(f_1\, ,f_2)\, \longmapsto\,
(-f_2\circ\sigma^{-1}\, , f_1\circ\sigma^{-1})$. We note that
$$
(\sigma^*\overline{T})\circ T\,=\, -\text{Id}_{
{\mathcal O}^{\oplus 2}_X}\, .
$$

Consider $b_0$ defined in \eqref{b}.
Elements of $H^1(X, \, L^\vee\otimes ({\mathcal O}^{\oplus
2}_X))^{\oplus b_0}$ parametrize extensions of the form
$$
0\,\longrightarrow\, ({\mathcal O}^{\oplus 2}_X)^{\oplus b_0}
\,\longrightarrow\, V\, \longrightarrow\, L\,\longrightarrow
\, 0\, ,
$$
where $L$ is the above line bundle on $X$.
There is a universal short exact sequence
\begin{equation}\label{cV}
0 \,\longrightarrow\, ({\mathcal O}^{\oplus 2}_{X\times
{\mathbb A}})^{\oplus b_0}
\,\longrightarrow\, {\mathcal V}\, \longrightarrow\,
p^*_1L\,\longrightarrow \, 0\, ,
\end{equation}
where ${\mathbb A}\, :=\, H^1(X, \, L^\vee\otimes
({\mathcal O}^{\oplus 2}_X))^{\oplus b_0}$, and $p_1$
is the projection of $X\times{\mathbb A}$ to $X$.
Since $d$ is sufficiently large, all stable vector bundles
over $X$ of rank $n$ and determinant $L$ occur in the family
$\mathcal V$ in \eqref{cV}. Let
\begin{equation}\label{cS}
{\mathcal S}\,\subset\,{\mathbb A}
\end{equation}
be the locus of stable bundles for the family
in \eqref{cV}; from the openness of the stability condition
(see \cite{Ma}) it follows that ${\mathcal S}$ is a Zariski open
subset. Let $U'_X(n,L)$ be the moduli space of stable vector
bundles $E$ over $X$ of rank $n$ with $\bigwedge^n E\,=\, L$. Let
\begin{equation}\label{Phi}
\Phi\, :\, {\mathcal S}\, \longrightarrow\, U'_X(n,L)
\end{equation}
be the surjective morphism representing the family $\mathcal V$ in
\eqref{cV}.

Let
\begin{equation}\label{AU}
\psi\, :\, U'_X(n,L)\, \longrightarrow\, U'_X(n,L)
\end{equation}
be the bijection defined by $E\, \longmapsto\, \sigma^*
\overline{E}$, where $\sigma$ is the isomorphism in \eqref{r-si},
and $\overline{E}$ is the vector bundle over $\overline{X}$ 
corresponding to the vector bundle $E$ over $X$. It should be
clarified that $\psi$ is not algebraic, it is not even holomorphic,
but anti-holomorphic. Let $\overline{U'_X(n,L)}$ be the variety
obtained from $U'_X(n,L)$ using the automorphism of the field
$\mathbb C$ defined by $z\, \longmapsto\, \overline{z}$. Therefore,
the complex points of $\overline{U'_X(n,L)}$ are in bijective
correspondence with the complex points of $U'_X(n,L)$. Using
this bijection, if we consider $\psi$ in \eqref{AU} as a map
$U'_X(n,L)\, \longrightarrow\, \overline{U'_X(n,L)}$, then this
map is an algebraic isomorphism.

The isomorphism $\psi$ in \eqref{AU} is clearly involutive.
This $\psi$ defines a real structure
on the complex variety
$U'_X(n,L)$. The corresponding variety over $\mathbb R$
is the moduli space $U'_Y(n,L)\, :=\, {\det}^{-1}(L)$, where
\begin{equation}\label{det}
\det\, :\, U'(n,d)\, \longrightarrow\, \text{Pic}^d(Y)
\end{equation}
is the morphism defined by $E\, \longmapsto\, \bigwedge^n E$,
and $L$ is the point in \eqref{fL}.

Let
\begin{equation}\label{th.}
\theta\, :\, H^1(X, \, L^\vee\otimes
({\mathcal O}^{\oplus 2}_X))^{\oplus b_0}\, \longrightarrow\,
H^1(X, \, L^\vee\otimes
({\mathcal O}^{\oplus 2}_X))^{\oplus b_0}
\end{equation}
be the conjugate linear involution constructed using
$\eta$ and $T$ defined in \eqref{eta} and \eqref{r-s3}
respectively. The subset $\mathcal S$ in \eqref{cS} is preserved
by $\theta$, and
\begin{equation}\label{2ph}
\psi\circ\Phi \, =\, \Phi\circ \theta\, ,
\end{equation}
where $\Phi$ and $\psi$ are constructed in \eqref{Phi} and
\eqref{AU} respectively. Therefore, the morphism $\Phi$
is defined over $\mathbb R$.

The fixed point locus
\begin{equation}\label{th.2}
{\mathbb A}^\theta\, \subset\, {\mathbb A}
\end{equation}
for $\theta$ is a $\mathbb R$--linear
subspace such that the natural homomorphism
$$
{\mathbb A}^\theta\otimes_{\mathbb R}{\mathbb C}
\,\longrightarrow\, H^1(X, \, L^\vee\otimes
({\mathcal O}^{\oplus 2}_X))^{\oplus b_0}
$$
is an isomorphism;
hence ${\mathbb A}^\theta$ is Zariski dense in $\mathbb A$.

Let
\begin{equation}\label{cW0}
{\mathcal W}\, \longrightarrow\, {\mathcal U} \times Y
\end{equation}
be a Poincar\'e sheaf, where ${\mathcal U}$ is a nonempty Zariski
open subset of $U'_Y(n,d)$. The morphism $\det$ in \eqref{det}
is an open smooth surjective morphism, hence
the image $\det ({\mathcal U})$ is a nonempty Zariski open subset
of $\text{Pic}^d(Y)$. We noted earlier that there is a
real point of $\text{Pic}^d(Y)$ which are not a line bundle
on $Y$ (see \eqref{fL}), and also it is known
that any (nonempty) connected component of the locus of
real points in a smooth quasiprojective variety is
Zariski dense. Hence the set of real points
of $\text{Pic}^d(Y)$ which are not line bundles on $Y$ is
Zariski dense in $\text{Pic}^d(Y)$. In particular, this
set intersects the Zariski open subset $\det({\mathcal U})$,
where ${\mathcal U}$ is the open subset in \eqref{cW0}.
Therefore, we may take the chosen point $L$ in \eqref{fL} to be
inside $\det({\mathcal U})$. Hence we assume that
$$
L\, \in\, \det({\mathcal U})\, .
$$

Since ${\mathcal S}$ in \eqref{cS} is a nonempty Zariski open
subset, and ${\mathbb A}^\theta$ defined in \eqref{th.2} is Zariski
dense, we know that ${\mathcal S}\bigcap {\mathbb A}^\theta$
is Zariski dense in ${\mathbb A}$.
Take any point
\begin{equation}\label{z}
z\, \in\, {\mathcal S}\cap {\mathbb A}^\theta
\end{equation}
such that the corresponding vector bundle
\begin{equation}\label{z2}
{\mathcal V}_z\, :=\,
{\mathcal V}\vert_{\{z\}\times X}
\end{equation}
lies in the open subset $\mathcal U$ in \eqref{cW0}.
Since $z\, \in\, {\mathbb A}^\theta$, from \eqref{2ph}
we know that $\psi(\Phi(z))\,=\, \Phi(z)$, where
$\psi$ and $\Phi$ are defined in \eqref{AU} and
\eqref{Phi} respectively. Therefore, there is an isomorphism
\begin{equation}\label{beta}
\beta\, :\,{\mathcal V}_z \, \longrightarrow\,
\sigma^*\overline{{\mathcal V}_z}\, ,
\end{equation}
constructed using
$\sigma$ and $T$ defined in \eqref{r-si} and \eqref{r-s3}
respectively, such that
\begin{equation}\label{inv}
(\sigma^*\overline{\beta})\circ\beta\,=\, -
\text{Id}_{{\mathcal V}_z}\, .
\end{equation}
This isomorphism $\beta$ fits in the following commutative
diagram
$$
\begin{matrix}
0 & \longrightarrow & ({\mathcal O}^{\oplus 2}_X)^{\oplus b_0}
& \longrightarrow & {\mathcal V}_z & \longrightarrow & L
& \longrightarrow & 0\\
&& ~\Big\downarrow{T} && ~\Big\downarrow\beta &&
~\Big\downarrow \eta\\
0 & \longrightarrow & \sigma^*({\mathcal O}^{\oplus
2}_{\overline{X}})^{\oplus b_0}= ({\mathcal O}^{\oplus
2}_X)^{\oplus b_0}
& \longrightarrow & \sigma^*\overline{{\mathcal V}_z} &
\longrightarrow & \sigma^*\overline{L} & \longrightarrow & 0
\end{matrix}
$$
where the horizontal exact sequences are as in \eqref{cV},
and the maps $T$ and $\eta$ are constructed in \eqref{r-s3}
and \eqref{eta} respectively.

Let $\underline{z}\, \in\, U'(n,d)$ be the point corresponding
to the vector bundle ${\mathcal V}_z$ in \eqref{z2}.
Restricting the Poincar\'e bundle $\mathcal W$ to
$\{\underline{z}\}\times Y$ we get a vector bundle over
$Y$ which is represented by this point $\underline{z}$
of the moduli space.

For any geometrically stable vector bundle $W_0$
over $Y$, the group of all automorphisms of the corresponding
vector bundle $W$ over $X$ is the group of nonzero complex numbers.
There is a natural isomorphism
$$
\gamma\, :\, W \, \longrightarrow\,
\sigma^*\overline{W}
$$
such that $(\sigma^*\overline{\gamma})\circ\gamma
\,=\, \text{Id}_{W}$. Since any other isomorphism
$W \, \longrightarrow\,\sigma^*\overline{W}$ must be of the form
$c\cdot \gamma$, where $c\, \in\, {\mathbb C}^*$ (recall that
the automorphisms of $W$ are the nonzero scalars), we conclude
that there is no isomorphism
$$
\phi\, :\, W \, \longrightarrow\,
\sigma^*\overline{W}
$$
such that $(\sigma^*\overline{\phi})\circ\phi
\,=\, - \text{Id}_{W}$ (there is no complex number such that
$\overline{c}c\,=\, -1$).

But the vector bundle ${\mathcal W}\vert_{\{\underline{z}\}
\times Y}$ in \eqref{cW0} contradicts the above observation that it
is not possible to have simultaneously isomorphisms $\gamma$ and
$\phi$ of the above type. From this contradiction we conclude
that there is no Poincar\'e sheaf over 
${\mathcal U}\times Y$. This completes the proof of the theorem.
\end{proof}

\end{document}